\documentclass{article}

\usepackage{amssymb}

\newtheorem{theorem}{Theorem}
\newtheorem{example}{Example}
\newtheorem{lemma}{Lemma}
\newtheorem{proposition}[theorem]{Proposition}
\newenvironment{proof}[1][Proof]{\noindent\textbf{#1.} }{\ \rule{0.5em}{0.5em}}

\DeclareSymbolFont{AMSb}{U}{msb}{m}{n}
\DeclareMathSymbol{\N}{\mathbin}{AMSb}{"4E}
\DeclareMathSymbol{\Z}{\mathbin}{AMSb}{"5A}
\DeclareMathSymbol{\R}{\mathbin}{AMSb}{"52}
\DeclareMathSymbol{\Q}{\mathbin}{AMSb}{"51}
\DeclareMathSymbol{\I}{\mathbin}{AMSb}{"49}
\DeclareMathSymbol{\C}{\mathbin}{AMSb}{"43}

\title{Estimating the $k$-th coefficient of $\left( f\left( z\right) \right)
^{n}$ when $k$ is not too large}

\author{Valerio De Angelis\\ Mathematics Department\\
Xavier University of Louisiana}
\date{}

\begin{document}
\maketitle



\begin{abstract}
% Text of abstract
We derive asymptotic estimates for the coefficient of $z^{k}$ in $\left(
f\left( z\right) \right) ^{n}$ when $n\rightarrow \infty $ and $k$ is of
order $n^{\delta }$, where $0<\delta <1,$ and $f\left( z\right) $ is a power
series satisfying suitable positivity conditions and with $f\left( 0\right)
\neq 0,$ $f^{\prime }\left( 0\right) =0.$ We also show that there is a
positive number $\varepsilon <1$ (easily computed from the pattern of non-zero
coefficients of $f\left( z\right) $) such that the same coefficient is
positive for large $n$ and $\varepsilon <\delta <1$, and admits an
asymptotic expansion in inverse powers of $k$. We use the asymptotic
estimates to prove that certain finite sums of exponential and trigonometric
functions are non-negative, and illustrate the results with examples.

\end{abstract}

% main text
\section{Introduction}
%\label{}
Let $f\left( z\right) $ be a power series in $z$ with real coefficients and
positive radius of convergence $R$. In this paper, we consider the problem of asymptotic
evaluation of the coefficients of large powers of $f(z)$. We will denote the coefficient of $x^k$ in 
$f(z)$ by $\widehat{f}\left( k\right) $, and the coefficient of $x^k$ in 
$\left(f(z)\right)^n$ by $\widehat{f^n}\left( k\right) $. As we will soon discuss, 
we are especially interested in the case that the parameter $k$ goes to infinity with $n$, but in
such a way that the ratio $k/n$ goes to zero.
Since multiplication of $f\left( z\right) $ by
a monomial $cz^{m}$ simply corresponds to replacing $\widehat{f}\left(
k\right) $ by $c\widehat{f}\left( k-m\right) $, we assume throughout that 
$f\left( 0\right) =\widehat{f}\left( 0\right) =1.$

The coefficients $\widehat{f}\left( k\right) $ can be represented by the
integral formula 
 \begin{equation}
\hat{f}(k)=\frac{1}{2\pi r^{k}}\int_{-\pi }^{\pi }f(re^{i\theta
})e^{-ik\theta }d\theta ,  \label{int}
 \end{equation}%
where $r$ is any positive number less than $R$. In fact, the right side of $%
\left( \ref{int}\right) $ is defined for every real number $k$, and so we
can define $\hat{f}(k)$ for real $k$ by the same formula.

Replacing $f\left( z\right) $ with $\left( f\left( z\right) \right) ^{n}$,
the study of the corresponding integral formula 
 \begin{equation}
\widehat{f^{n}}(k)=\frac{1}{2\pi r^{k}}\int_{-\pi }^{\pi }\left(
f(re^{i\theta })\right) ^{n}e^{-ik\theta }d\theta   \label{int1}
 \end{equation}%
when $n$ is large provides asymptotic estimates for the coefficients of
large powers of $f\left( z\right) .$ 

The freedom of choice for $r$ in such integral representations provides a
powerful analytic method for many asymptotic evaluations, and this approach has
been successfully used by several authors \cite{MW1}, \cite{MW2}, \cite{Hay}%
, \cite{OR}. See also the survey articles \cite{Od}, \cite{Od2} (containing
an extensive bibliography), and the book  \cite[Ch. VIII]{Flaj}. The
article \cite{Gardy} reviews results in the case that $n$ and $k$ remain
roughly proportional, and under the assumptions that all coefficients of $%
f\left( z\right) $ are positive, and $f\left( 0\right) \neq 0$, $f^{\prime
}\left( 0\right) \neq 0$. The same article presents extensions to the case
that $k$ and $n$ have more general relationships, and when a multiplicative
factor $\psi \left( z\right) $ is introduced (so that the coefficients of $%
\psi \left( z\right) \left( f\left( z\right) \right) ^{n}$ are the object of
study). 

When all coefficients  of $f(z)$ are positive, the main contribution to the 
integral $\left( \ref{int}\right) $ for large $n$ comes from a small neighborhood of $\theta =0$,
because the function $\theta \mapsto \left\vert f(re^{i\theta
})\right\vert $ has an absolute maximum there. If $f'(0)=0$, there will be other relative maxima 
of the same function in $[-\pi,\pi]$, but the corresponding contributions to the integral
will be negligible {\em provided} $k$ remains large enough when $n$ grows.

The present article addresses the problem of deriving full asymptotic expansions (or
upper bound estimates when a full expansion is not possible) for the 
coefficients $\widehat{f}\left( k\right) $ when the
ratio $k/n$ is in an especially delicate range: large enough so that the
integral representation method is still useful, but not large enough to
neglect the contributions to the integral $\left( \ref{int}\right) $ arising
from the relative maxima of $\theta \mapsto \left\vert f(re^{i\theta
})\right\vert $ at points other than $\theta =0.$
The condition that all coefficients be positive will be replaced by the the following weaker condition.
We say that $f$ is \emph{strongly
positive at} $r$ if%
 \[
|f(re^{i\theta })|<f(r)\ \ \ \mbox{for}\ 0<\theta <2\pi
 \]%
and we say that $f$ is \emph{strongly positive }if it is strongly positive at every $r$. If $f$
is strongly positive at $r$, we define $\mu _{f}(r)=rf^{\prime }(r)/f(r)$,
and $\sigma _{f}\left( r\right) =r\mu _{f}^{\prime }(r).$

When $f\left( z\right) $ is a polynomial, it is easy to check that $\mu
_{f}(r)\rightarrow 0$ as $r\rightarrow 0$ (remember our standing assumption
that $f\left( 0\right) =1$), and $\mu _{f}(r)\rightarrow \deg f$ as $%
r\rightarrow \infty $. It is proved in \cite[Theorem 6.4]{pcp} that strong
positivity on an open interval containing $r$ is enough to ensure that $%
\sigma _{f}\left( r\right) >0$. Hence for strongly positive polynomials, $%
\mu _{f}$ provides a bijection from $\left( 0,\infty \right) $ to $\left(
0,\deg f\right) $. Since by definition $\widehat{f}\left( k\right) =0$ for $%
k\notin \left[ 0,\deg f\right] $ (and the boundary points $0$ and $\deg f$
are trivial), we can estimate all coefficients of $\left( f\left( z\right)
\right) ^{n}$ if we can estimate $\widehat{f^{n}}\left( n\mu _{f}(r)\right) $
for all $r\in \left( 0,\infty \right) $.

In the general case of a strongly positive power series with radius of
convergence $R$, $\mu _{f}$ is still strictly increasing on $\left(
0,R\right) $, and $\mu _{f}(r)\rightarrow 0$ as $r\rightarrow 0$. 
One of the results of \cite[Theorem 4.1]{ae} is to show that 
$\widehat{f^{n}}\left( n\mu _{f}(r)\right) $ has an asymptotic expansion with dominant term 
$\frac{\left( f(r)\right) ^{n}}{r^{n\mu _{f}(r)}\sqrt{%
2\pi n\sigma _{f}\left( r\right) }}$ as $n\rightarrow \infty$ (where $r$ is defined by
$\mu_f(r)=k/n$), but only under the 
assumption that either $k/n$ remains bounded away from zero, or $f'(0)\neq 0$.
In the present work, we investigate the case 
$f^{\prime }\left( 0\right) =0$ and $%
k/n\rightarrow 0$. We find that there is a positive number $\epsilon
_{0}=\epsilon _{0}\left( f\right) <1$ (easily calculated from the pattern of
non-zero coefficients of $f$) such that $\widehat{f^{n}}\left( k\right) $ has a 
similar asymptotic expansion, with same dominant term, whenever $k$ is of
order $n^{\delta }$ for some $\delta \in \left( \epsilon _{0},1\right) .$

If $k$ is of order $n^{\delta }$ with $0<\delta \leq \epsilon _{0}$, then
the problem becomes more difficult, and the initial pattern of non-zero
coefficients plays an increasingly important role as $\delta $ decreases. An
asymptotic expression for the coefficients $\widehat{f^{n}}(k)$ when $n$ is
large can still be found. However, in some cases the expression found this
way vanishes, and then we only obtain an upper bound for the asymptotic size
of the coefficient. In this case, it may be possible to obtain the exact
asymptotic beahvior by considering higher order terms, but this would require
considerable more work.

As an unintended consequence of the asymptotic estimates, we obtain a
curious result on the non-negativity of certain finite sums involving
trigonometric and exponential functions that we cannot prove in any other
way. We also provide examples with $0<\delta \leq \epsilon _{0}$ illustrating both the
case in which the expression derived is strictly positive, and the case in
which it vanishes. While the case $k$ bounded can easily be treated
with Stirling's formula, if $k$ goes to infinity slower than a power of $n$,
the methods of this paper do not provide any estimate for the coefficients $%
\widehat{f^{n}}(k)$, and the problem of their estimate in that range remains
open.

This paper complements and extends the work done in \cite{ineq} and \cite{ae}%
. For the reader's convenience, in the next section we summarize all the
results from those papers needed here.

\section{Summary of needed results from previous work}

We will assume throughout that $f$ is a non-constant power series with
positive radius of convergence, normalized such that $f\left( 0\right) =1$,
and strongly positive at all sufficiently small $r>0.$ We define $L\left(
f\right) =\left\{ k>0:\widehat{f}\left( k\right) \neq 0\right\} .$ If $\gcd
L\left( f\right) =d>1$, we can write $f\left( z\right) =g\left( z^{d}\right) 
$ for some power series $g$ such that $\gcd L\left( g\right) =1$, and then $%
\widehat{f}\left( k\right) =0$ if $k$ is not a multiple of $d$, while $%
\widehat{f}\left( dk\right) =\widehat{g}\left( k\right) $. Hence, without
losing generality, we will assume that $\gcd L\left( f\right) =1$.

Let $l=l\left( f\right) =\min L\left( f\right) $. Then our standing condition $%
f^{\prime }\left( 0\right) =0$ is equivalent to $l>1.$
If $m=\left[ l/2\right] $ is the greatest integer not
exceeding $l/2$, and $1\leq j\leq m$, there is some integer $k\in L\left(
f\right) $ such that $jk$ is not a multiple of $l.$ We define $l_{j}=\min
\left\{ k\in L\left( f\right) :jk\not\equiv 0\pmod{l}
\right\} ,$ \ $1\leq j\leq m,$ and $l_{0}=l$ (the set $\left\{ l_{j}:0\leq
j\leq m\right\} $ of exponents was first considered by Hayman in \cite{Hay2}).

The results from previous work that will be needed here are labeled (I)--(VI) below. We remark that in
order to derive only the asymptotic estimates of Theorem $\ref{T1}$ and the
subsequent examples (in other words, if the full asymptotic expansion of
Theorem \ref{T0} is not needed), only (I)--(V) are required. Since (I)
and (II) amount to easily checked estimates, all the required background for
the asymptotic estimates and examples are contained in \cite{ineq}. While
that paper treats the polynomial case, all the results quoted below remain
valid for power series, with virtually identical proofs.

\begin{enumerate}
\item[(I)] [See Lemma 5.2 of \cite{ae}, or Section 4 of \cite{pcp}] As $%
r\rightarrow 0^{+}$, we have
\begin{enumerate}
\item[(a)]%
  \[
\mu _{f}\left( r\right) =l\widehat{f}\left( l\right) r^{l}+O\left(
r^{l+1}\right) \mbox{,}  \label{a}
  \]%
and%
\item[(b)]
  \[
\sigma _{f}\left( r\right) =l^{2}\widehat{f}\left( l\right) r^{l}+O\left(
r^{l+1}\right) .  \label{b}
  \]
\end{enumerate}
\item[(II)] If we define the function $\phi \left( r,z\right) $ on the set 
 \[
V=\left\{ \left( r,z\right) \in \left[ 0,\infty \right) \times \C:%
\mbox{Re}f\left( re^{z}\right) >0\right\}
 \]%
by the equation%
 \[
\frac{f\left( re^{z}\right) }{f\left( r\right) }=\exp \left( \mu _{f}\left(
r\right) z+\frac{1}{2}\sigma _{f}\left( r\right) z^{2}\phi \left(
r,z\right) \right)
 \]%
for $z\neq 0$, $r>0$, and 
\begin{eqnarray*}
\phi \left( r,0\right) &=&1, \\
\phi \left( 0,z\right) &=&2\frac{e^{lz}-1-lz}{\left( lz\right) ^{2}}%
=1+2\sum_{k=0}^{\infty }\frac{\left( lz\right) ^{k+1}}{\left( k+3\right) !},
\end{eqnarray*}%
then we can write 
 \[
\phi \left( r,z\right) =2\frac{e^{lz}-1-lz}{\left( lz\right) ^{2}}%
+rzH\left( r,z\right) ,
 \]%
where $H\left( r,z\right) $ is a power series in $r$ and $z$ (this the
version of Lemma 5.3 of \cite{ae} for the case $l>1$).

\item[(III)] Strong positivity of $f$ at all sufficiently small $r>0$ is
equivalent to the condition $\widehat{f}\left( l_{j}\right) >0$ for $0\leq
j\leq m$ (Theorem 1 of \cite{ineq}, or Theorem II of \cite{Hay2}).

\item[(IV)] There is some $r_{0}>0$ such that for each $r\in \left( 0,r_{0}%
\right] $, the function $\theta \mapsto \left\vert f\left( re^{i\theta
}\right) \right\vert $ has precisely $m$ local maxima in the interval $%
\left( 0,\pi \right] $, and if we denote them by $\widetilde{\theta }%
_{j}\left( r\right) $, $1\leq j\leq m$, then the functions $r\mapsto $ $%
\widetilde{\theta }_{j}\left( r\right) $ are continuous on $\left( 0,r_{0}%
\right] $, and $\lim\limits_{r\rightarrow 0^{+}}\widetilde{\theta }%
_{j}\left( r\right) =\frac{2\pi j}{l}.$ If $l$ is even, then $\widetilde{%
\theta }_{m}\left( r\right) =\pi $ for all $r\in \left( 0,r_{0}\right] $.
(Proposition 2 of \cite{ineq}).

\item[(V)] Let $\theta _{j}=\frac{2\pi j}{l}$, $1\leq j\leq m$. If $l>2$,
define%
 \[
l_{j}^{\prime }=\min \left\{ k\in L\left( f\right) :2jk\not\equiv 0\pmod{l} \right\} ,\mbox{ \ }1\leq j<l/2.
 \]%
(Note that by definition of $l_{j}\,$and $l_{j}^{\prime }$, we have $%
l_{j}^{\prime }\geq l_{j},$ $\exp \left( il_{j}\theta _{j}\right) \neq 1$, $%
\sin \left( \theta _{j}l_{j}^{\prime }\right) \neq 0$, and if $l$ is odd,
then $l_{j}^{\prime }=l_{j}$). Then, if $l>2$ and $j<l/2,$ we have (Lemma 3
of \cite{ineq}):%
 \[
\widetilde{\theta }_{j}\left( r\right) =\theta _{j}-r^{l_{j}^{\prime }-l}%
\frac{l_{j}^{\prime }\widehat{f}\left( l_{j}^{\prime }\right) }{l^{2}%
\widehat{f}\left( l\right) }\sin \left( \theta _{j}l_{j}^{\prime }\right)
\left( 1+O\left( r\right) \right) \mbox{ \ as }r\rightarrow 0^{+}.
 \]

\item[(VI)] \label{prop} [Proposition 2.2 of \cite{ae}] Suppose that $\phi
(z)$ is analytic in a neighborhood of $z=0$, with $\phi (0)=1$, and let $%
\varepsilon >0$ be such that $|\phi (z)-1|\leq 1/2$ for $|z|\leq
2\varepsilon $. Let 
 \[
G(\lambda )=\int_{-\varepsilon }^{\varepsilon }e^{-\lambda ^{2}\theta
^{2}\phi (i\theta )}d\theta .
 \]%
Then as $\lambda \rightarrow \infty $, $G\left( \lambda \right) $ has the
asymptotic expansion%
 \[
G(\lambda )\sim \frac{\sqrt{\pi }}{\lambda }\left( 1+\sum_{\nu =1}^{\infty
}a_{\nu }\lambda ^{-2\nu }\right) \mbox{, where \ }a_{\nu }=\frac{(-1)^{\nu }}{%
4^{\nu }\nu !}\left[ \frac{d ^{2\nu }}{d z^{2\nu }}\left( \phi
(z)\right) ^{-\nu -1/2}\right] _{z=0}.
 \]
\end{enumerate}

\section{Expansions near the local maxima}

We begin by further normalizing $f\left( z\right) $. If $g\left( z\right)
=f\left( az\right) $, then clearly $\widehat{g}\left( k\right) =a^{k}%
\widehat{f}\left( k\right) $. Hence no generality is lost if we choose a
suitable $a$ that will simplify our formulas, and in view of (I) of the previous
section, we will from now on assume that $l\widehat{f}\left( l\right) =1.$
So we consider powers series $f\left( z\right) $ of form $f\left( z\right)
=1+\frac{1}{l}x^{l}+\widehat{f}\left( l+1\right) z^{l+1}+\cdots .$

The next lemma is a refinement of Proposition 4 of \cite{ineq}.

\begin{lemma}
\label{L1}Let $1\leq j\leq m$, and let $\widetilde{\theta }_{j}\left(
r\right) $, $\theta _{j}$ be as in (IV) and (V). Define%
 \[
A_{j}=\widehat{f}\left( l_{j}\right) \left( e^{il_{j}\theta _{j}}-1\right) ,
 \]%
Then we have%
 \[
e^{-i\widetilde{\theta }_{j}\left( r\right) }\frac{f\left( re^{i\widetilde{%
\theta }_{j}\left( r\right) }\right) }{f\left( r\right) }=\exp \left( -i\mu
\theta _{j}+A_{j}r^{l_{j}}+O\left( r^{l_{j}+1}\right) \right) \mbox{ \ as }%
r\rightarrow 0^{+}.
 \]
\end{lemma}

\begin{proof}
Define%
 \[
B_{j}=\left\{ 
\begin{tabular}{ll}
$0$ & $\mbox{if }l\mbox{ is even and }j=m,$ \\ 
$\frac{l_{j}^{\prime }\widehat{f}\left( l_{j}^{\prime }\right) }{l}\sin
\left( \theta _{j}l_{j}^{\prime }\right) $ & $\mbox{in all other cases.}$%
\end{tabular}%
\right.
 \]%

Using (V) above, we find 
\begin{eqnarray*}
&&f\left( re^{i\widetilde{\theta }_{j}\left( r\right) }\right) \\
&=&\sum_{k}\widehat{f}\left( k\right) r^{k}\exp \left( ik\left( \theta
_{j}-B_{j}r^{l_{j}^{\prime }-l}+O\left( r^{l_{j}^{\prime }-l+1}\right)
\right) \right) \\
&=&\sum_{k<l_{j}}\widehat{f}\left( k\right) r^{k}\exp \left( ik\left(
-B_{j}r^{l_{j}^{\prime }-l}+O\left( r^{l_{j}^{\prime }-l+1}\right) \right)
\right) \\
&&+\widehat{f}\left( l_{j}\right) e^{il_{j}\theta _{j}}r^{l_{j}}+O\left(
r^{l_{j}+1}\right) \mbox{ \ \ \ (because }e^{ik\theta _{j}}=1\mbox{ for }%
k<l_{j}\mbox{, and }l_{j}^{\prime }-l\geq 1\mbox{)} \\
&=&1+\sum_{l\leq k<l_{j}}\widehat{f}\left( k\right) r^{k}\left(
1-ikB_{j}r^{l_{j}^{\prime }-l}+O\left( r^{l_{j}^{\prime }-l+1}\right) \right)
\\
&&+\widehat{f}\left( l_{j}\right) e^{il_{j}\theta _{j}}r^{l_{j}}+O\left(
r^{l_{j}+1}\right) \mbox{ \ \ \ \ (because }\widehat{f}\left( 0\right) =1%
\mbox{, and }\widehat{f}\left( k\right) =0\mbox{ for }0<k<l\mbox{)} \\
&=&1+\sum_{l\leq k<l_{j}}\widehat{f}\left( k\right)
r^{k}-iB_{j}r^{l_{j}^{\prime }}+O\left( r^{l_{j}^{\prime }+1}\right) +%
\widehat{f}\left( l_{j}\right) e^{il_{j}\theta _{j}}r^{l_{j}}+O\left(
r^{l_{j}+1}\right) \\
&=&\sum_{0\leq k<l_{j}}\widehat{f}\left( k\right)
r^{k}-iB_{j}r^{l_{j}^{\prime }}+\widehat{f}\left( l_{j}\right)
e^{il_{j}\theta _{j}}r^{l_{j}}+O\left( r^{l_{j}+1}\right) \mbox{ \ \
(because }l_{j}^{\prime }\geq l_{j}\mbox{).}
\end{eqnarray*}%
Since 
 \[
f\left( r\right) =\sum\limits_{0\leq k<l_{j}}\widehat{f}\left( k\right)
r^{k}+\widehat{f}\left( l_{j}\right) r^{l_{j}}+O\left( r^{l_{j}+1}\right) 
\mbox{ \ and }\left( \sum\limits_{0\leq k<l_{j}}\widehat{f}\left( k\right)
r^{k}\right) ^{-1}=1+O\left( r^{l}\right) ,
 \]%
we find%
\begin{eqnarray*}
\frac{f\left( re^{i\widetilde{\theta }_{j}\left( r\right) }\right) }{%
f\left( r\right) } &=&\frac{\sum\limits_{0\leq k<l_{j}}\widehat{f}\left(
k\right) r^{k}-iB_{j}r^{l_{j}^{\prime }}+\widehat{f}\left( l_{j}\right)
e^{il_{j}\theta _{j}}r^{l_{j}}+O\left( r^{l_{j}+1}\right) }{%
\sum\limits_{0\leq k<l_{j}}\widehat{f}\left( k\right) r^{k}+\widehat{f}%
\left( l_{j}\right) r^{l_{j}}+O\left( r^{l_{j}+1}\right) } \\
&=&\frac{1+\left( \widehat{f}\left( l_{j}\right) e^{il_{j}\theta
_{j}}r^{l_{j}}-iB_{j}r^{l_{j}^{\prime }}+O\left( r^{l_{j}+1}\right) \right)
\left( \sum\limits_{0\leq k<l_{j}}\widehat{f}\left( k\right) r^{k}\right)
^{-1}}{1+\left( \widehat{f}\left( l_{j}\right) r^{l_{j}}+O\left(
r^{l_{j}+1}\right) \right) \left( \sum\limits_{0\leq k<l_{j}}\widehat{f}%
\left( k\right) r^{k}\right) ^{-1}} \\
&=&\frac{1-iB_{j}r^{l_{j}^{\prime }}+\widehat{f}\left( l_{j}\right)
e^{il_{j}\theta _{j}}r^{l_{j}}+O\left( r^{l_{j}+1}\right) }{1+\widehat{f}%
\left( l_{j}\right) r^{l_{j}}+O\left( r^{l_{j}+1}\right) } \\
&=&\left( 1-iB_{j}r^{l_{j}^{\prime }}+\widehat{f}\left( l_{j}\right)
e^{il_{j}\theta _{j}}r^{l_{j}}+O\left( r^{l_{j}+1}\right) \right) \\
&&\cdot \left( 1-\widehat{f}\left( l_{j}\right) r^{l_{j}}+O\left(
r^{l_{j}+1}\right) \right) \\
&=&1+\widehat{f}\left( l_{j}\right) \left( e^{il_{j}\theta _{j}}-1\right)
r^{l_{j}}-iB_{j}r^{l_{j}^{\prime }}+O\left( r^{l_{j}+1}\right) \\
&=&\exp \left( A_{j}r^{l_{j}}-iB_{j}r^{l_{j}^{\prime }}+O\left(
r^{l_{j}+1}\right) \right) .
\end{eqnarray*}%
On the other hand, we have from (V) (or (IV) if $l$ is even and $j=m$) and
(I)(a):%
\begin{eqnarray*}
e^{-i\widetilde{\theta }_{j}\left( r\right) } &=&\exp \left( -i\mu \theta
_{j}+i\mu B_{j}r^{l_{j}^{\prime }-l}\left( 1+O\left( r\right) \right) \right)
\\
&=&\exp \left( -i\mu \theta _{j}+iB_{j}r^{l_{j}^{\prime }}+O\left(
r^{l_{j}^{\prime }+1}\right) \right)
\end{eqnarray*}%
and the result follows.
\end{proof}

\begin{lemma}
\label{L2}Let $1\leq j\leq m$. Then there is some $r_{1}>0$ and a continuous
function $F\left( r,\theta \right) $ defined for $0\leq \theta \leq 2\pi $
and $0\leq r\leq r_{1}$ such that $F\left( r,0\right) =0$, and 
 \[
\frac{f\left( re^{i\widetilde{\theta }_{j}\left( r\right) }e^{i\theta
}\right) }{f\left( re^{i\widetilde{\theta }_{j}\left( r\right) }\right) }=%
\frac{f\left( re^{i\theta }\right) }{f\left( r\right) }\exp \left(
r^{l_{j}}F\left( r,\theta \right) \right) .
 \]
\end{lemma}

\begin{proof}
Clearly if the above equation defines $F\left( r,\theta \right) $ for $r>0$,
then $F\left( r,0\right) =0$. Hence we only need to show that $F\left(
r,\theta \right) $ is bounded as $r\rightarrow 0^{+}.$ From (IV) and (V), we
have $\widetilde{\theta }_{j}\left( r\right) =\theta _{j}+O\left(
r^{l_{j}^{\prime }-l}\right) $. Define $g\left( r\right) =re^{i\widetilde{%
\theta }_{j}\left( r\right) }.$ If $k\in L\left( f\right) $ and $0<k<l_{j}$,
then $e^{ik\theta _{j}}=1$ and so%
 \[
\left( g\left( r\right) \right) ^{k}=r^{k}e^{ik\theta _{j}}\left( 1+O\left(
r^{l_{j}^{\prime }-l}\right) \right) =r^{k}+O\left( r^{l_{j}^{\prime
}}\right) =r^{k}+O\left( r^{l_{j}}\right) ,
 \]%
because $k\geq l$ and $l_{j}^{\prime }\geq l_{j}.$ So we find%
\begin{eqnarray*}
f\left( g\left( r\right) e^{i\theta }\right) &=&1+\sum\limits_{0<k<l_{j}}%
\widehat{f}\left( k\right) \left( g\left( r\right) \right) ^{k}e^{ik\theta
}+\sum\limits_{k\geq l_{j}}\widehat{f}\left( k\right) \left( g\left(
r\right) \right) ^{k}e^{ik\theta } \\
&=&1+\sum\limits_{0<k<l_{j}}\widehat{f}\left( k\right) r^{k}e^{ik\theta
}+O\left( r^{l_{j}}\right) +\sum\limits_{k\geq l_{j}}\widehat{fe}\left(
k\right) \left( g\left( r\right) \right) ^{k}e^{ik\theta }.
\end{eqnarray*}

Note that the estimate $O\left( r^{l_{j}}\right) $ in the last equation is
uniform in $\theta $, and similarly the last sum is $O\left(
r^{l_{j}}\right) $ uniformly in $\theta $. Moreover, extending the first sum
to all $k$ only adds a \ term $O\left( r^{l_{j}}\right) $, again uniformly
in $\theta $. In other words, we can find some $r_{1}>0$ and a continuous
function $G\left( r,\theta \right) $ on $\left[ 0,r_{1}\right] \times \left[
0,2\pi \right] $ so that%
 \[
f\left( g\left( r\right) e^{i\theta }\right) =f\left( re^{i\theta }\right)
+r^{l_{j}}G\left( r,\theta \right) .
 \]

Hence we find%
\begin{eqnarray*}
\frac{f\left( g\left( r\right) e^{i\theta }\right) }{f\left( g\left(
r\right) \right) } &=&\frac{f\left( re^{i\theta }\right) +r^{l_{j}}G\left(
r,\theta \right) }{f\left( r\right) +r^{l_{j}}G\left( r,0\right) } \\
&=&\frac{f\left( re^{i\theta }\right) \left[ 1+r^{l_{j}}G\left( r,\theta
\right) \left( f\left( re^{i\theta }\right) \right) ^{-1}\right] }{f\left(
r\right) \left[ 1+r^{l_{j}}G\left( r,0\right) \left( f\left( r\right)
\right) ^{-1}\right] }
\end{eqnarray*}%
and the result follows (by decreasing $r_{1}$ if necessary) because $f\left(
re^{i\theta }\right) =1+O\left( r\right) $ uniformly in $\theta $.
\end{proof}

\section{Preliminary estimates}

Now take $r_{0}>0$ and $\varepsilon >0$ small enough so that the intervals $%
\left( \widetilde{\theta }_{j}\left( r\right) -\varepsilon ,\widetilde{%
\theta }_{j}\left( r\right) +\varepsilon \right) $, $\left( 0,\varepsilon
\right) $ are all disjoint for $1\leq j\leq m$ and $0<r\leq r_{0}$. We make
the following definitions, for $0<r\leq r_{0},$ $n\geq 1$, and any integer $%
k $:%
 \[
\sigma =\sigma _{f}\left( r\right) ,
 \]%
 \[
\lambda =\sqrt{\frac{n\sigma }{2}}
 \]%
 \[
I_{0}=\sqrt{\frac{2n\sigma }{\pi }}\int\limits_{0}^{\varepsilon }\left( 
\frac{f\left( re^{i\theta }\right) }{f\left( r\right) }\right)
^{n}e^{-ik\theta }d\theta ,
 \]%
 \[
I_{j}=\left\{ 
\begin{tabular}{ll}
$\sqrt{\frac{2n\sigma }{\pi }}\int\limits_{\pi -\varepsilon }^{\pi }\left( 
\frac{f\left( re^{i\theta }\right) }{f\left( r\right) }\right)
^{n}e^{-ik\theta }d\theta $ & $\mbox{if \ }l\mbox{ \ is even and }j=m,$ \\ 
$\sqrt{\frac{2n\sigma }{\pi }}\int\limits_{\widetilde{\theta }_{j}\left(
r\right) -\varepsilon }^{\widetilde{\theta }_{j}\left( r\right) +\varepsilon
}\left( \frac{f\left( re^{i\theta }\right) }{f\left( r\right) }\right)
^{n}e^{-ik\theta }d\theta $ & $\mbox{ \ in all other cases,}$%
\end{tabular}%
\right.
 \]%
 \[
W=\left[ 0,\pi \right] \diagdown \left( \left[ 0,\varepsilon \right] \cup
\bigcup\limits_{j=1}^{m}\left[ \widetilde{\theta }_{j}\left( r\right)
-\varepsilon ,\widetilde{\theta }_{j}\left( r\right) +\varepsilon \right]
\right) ,
 \]%
 \[
J=\sqrt{\frac{2n\sigma }{\pi }}\int\limits_{W}\left( \frac{f\left(
re^{i\theta }\right) }{f\left( r\right) }\right) ^{n}e^{-ik\theta }d\theta .
 \]

Since the coefficients $\widehat{f}\left( k\right) $ are real, we can write%
\begin{eqnarray}
\nonumber
r^{k}\sqrt{2\pi n\sigma }\frac{\widehat{f^{n}}\left( k\right) }{\left(
f\left( r\right) \right) ^{n}} &=&r^{k}\sqrt{2\pi n\sigma }\frac{1}{2\pi
r^{k}}\int\limits_{-\pi }^{\pi }\left( \frac{f\left( re^{i\theta }\right) 
}{f\left( r\right) }\right) ^{n}e^{-ik\theta }d\theta    \\ \nonumber
&=&r^{k}\sqrt{2\pi n\sigma }\frac{1}{\pi r^{k}}\mbox{Re}\int\limits_{0}^{%
\pi }\left( \frac{f\left( re^{i\theta }\right) }{f\left( r\right) }\right)
^{n}e^{-ik\theta }d\theta    \\  
&=&\mbox{Re}\left( I_{0}+\sum_{j=1}^{m}I_{j}+J\right) .  \label{sum}
\end{eqnarray}

We write $g\left( n\right) \approx h\left( n\right) $ to mean that $g\left(
n\right) =O\left( h\left( n\right) \right) $ and $h\left( n\right) =O\left(
g\left( n\right) \right) $ as $n\rightarrow \infty $. Equivalently, there is
a constant $c\geq 1$ such that $\frac{1}{c}h\left( n\right) \leq g\left(
n\right) \leq ch\left( n\right) $ for all large $n.$ We also write $g\left(
n\right) \simeq h\left( n\right) $ to mean that $g\left( n\right) /h\left(
n\right) \rightarrow 1$ as $n\rightarrow \infty .$ Both $\approx $ and $%
\simeq $ are equivalence relations.

Our first lemma in this section shows that as long as $k\rightarrow \infty $
like $n^{\delta }$ for some positive $\delta <1$, the integral $J$ is quite
small. Recall that $\phi \left( r,z\right) $ defined in (II) approaches $1$
uniformly in $r$ as $z\rightarrow 0$. So by decreasing $\varepsilon $ if
necessary, we can make sure that $\left\vert \phi \left( r,i\theta \right)
-1\right\vert \leq 1/2$ for $\left\vert \theta \right\vert \leq \varepsilon
, $ $0<r\leq r_{0}$.

\begin{lemma}
\label{L0}Suppose that $0<\delta <1$, $k\approx n^{\delta },$ and $r>0$ is
defined by $k=n\mu _{f}\left( r\right) $. Then $J=o\left( \lambda
^{-N}\right) $ for every $N>0,$ in the sense that $\lambda ^{N}J\rightarrow
0 $ as $n\rightarrow \infty .$
\end{lemma}

\begin{proof}
Note that $r=\mu _{f}^{-1}\left( k/n\right) \rightarrow 0$ as $n\rightarrow
\infty $. We find, for large $n$,%
\begin{eqnarray*}
\left\vert J\right\vert &=&\sqrt{\frac{2n\sigma }{\pi }}\left\vert
\int\limits_{W}\left( \frac{f\left( re^{i\theta }\right) }{f\left(
r\right) }\right) ^{n}e^{-ik\theta }d\theta \right\vert \leq \sqrt{\frac{%
2n\sigma }{\pi }}\int\limits_{W}\left\vert \frac{f\left( re^{i\theta
}\right) }{f\left( r\right) }\right\vert ^{n}d\theta \\
&\leq &\pi \sqrt{\frac{2n\sigma }{\pi }}\max \left\{ \left\vert \frac{%
f\left( re^{i\theta }\right) }{f\left( r\right) }\right\vert ^{n}:\theta \in
W\right\} ,
\end{eqnarray*}%
and $\left\vert f\left( re^{i\theta }\right) \right\vert $ achieves its
maximum on $W$ at one of the points $\varepsilon ,\widetilde{\theta }%
_{1}\left( r\right) \pm \varepsilon ,\ldots ,\widetilde{\theta }_{m}\left(
r\right) \pm \varepsilon $. By our choice of $\varepsilon $, we find 
 \[
\left\vert \frac{f\left( re^{i\varepsilon }\right) }{f\left( r\right) }%
\right\vert ^{n}=\left\vert \exp \left( -\varepsilon ^{2}\lambda ^{2}\phi
\left( r,i\varepsilon \right) \right) \right\vert \leq \exp \left(
-\varepsilon ^{2}\lambda ^{2}/2\right) .
 \]%
From (I)(b), $\lambda ^{2}\approx n\sigma \approx nr^{l}$, and using Lemma %
\ref{L1} and Lemma \ref{L2},%
\begin{eqnarray*}
\left\vert \frac{f\left( re^{i\left( \widetilde{\theta }_{j}\left( r\right)
\pm \varepsilon \right) }\right) }{f\left( r\right) }\right\vert ^{n}
&=&\left\vert \frac{f\left( re^{i\left( \widetilde{\theta }_{j}\left(
r\right) \pm \varepsilon \right) }\right) }{f\left( re^{i\widetilde{\theta }%
_{j}\left( r\right) }\right) }\right\vert ^{n}\left\vert \frac{f\left( re^{i%
\widetilde{\theta }_{j}\left( r\right) }\right) }{f\left( r\right) }%
\right\vert ^{n} \\
&=&\left\vert \frac{f\left( re^{i\varepsilon }\right) }{f\left( r\right) }%
\right\vert ^{n}\left\vert \exp \left( nr^{l_{j}}F\left( r,\theta \right)
+nO\left( r^{l_{j}}\right) \right) \right\vert \\
&\leq &\exp \left( -\varepsilon ^{2}\lambda ^{2}/2+bnr^{l_{j}}\right) \leq
\exp \left( -anr^{l}+bnr^{l_{j}}\right)
\end{eqnarray*}%
where $a$ and $b$ are positive constants. Since $l_{j}>l$, we can find a
constant $c>0$ such that the last expression is bounded by $e^{-c\lambda
^{2}}$, and the result follows.
\end{proof}

Lemma \ref{L3} below shows that $I_{j}$ is quite small when $k$ is large
enough.

\begin{lemma}
\label{L3}Suppose that $1\leq j\leq m,$ $1-\frac{l}{l_{j}}<\delta <1$, $%
k\approx n^{\delta },$ and $r>0$ is defined by $k=n\mu _{f}\left( r\right) $%
. Then $I_{j}=o\left( \lambda ^{-N}\right) $ for every $N>0.$
\end{lemma}

\begin{proof}
Note that $r=\mu _{f}^{-1}\left( k/n\right) \rightarrow 0$ as $n\rightarrow
\infty $. By definition of $\widetilde{\theta }_{j}\left( r\right) $, and
using Lemma \ref{L1}, we find, for large $n$ (when either $l$ is odd or $%
j\neq m)$,%
\begin{eqnarray*}
\left\vert I_{j}\right\vert &=&\sqrt{\frac{2n\sigma }{\pi }}\left\vert
\int\limits_{\widetilde{\theta }_{j}\left( r\right) -\varepsilon }^{%
\widetilde{\theta }_{j}\left( r\right) +\varepsilon }\left( \frac{f\left(
re^{i\theta }\right) }{f\left( r\right) }\right) ^{n}e^{-ik\theta }d\theta
\right\vert \\
&\leq &\sqrt{\frac{2n\sigma }{\pi }}\int\limits_{\widetilde{\theta }%
_{j}\left( r\right) -\varepsilon }^{\widetilde{\theta }_{j}\left( r\right)
+\varepsilon }\left\vert \frac{f\left( re^{i\theta }\right) }{f\left(
r\right) }\right\vert ^{n}d\theta \leq 2\varepsilon \sqrt{\frac{2n\sigma }{%
\pi }}\left\vert \frac{f\left( re^{i\widetilde{\theta }_{j}\left( r\right)
}\right) }{f\left( r\right) }\right\vert ^{n} \\
&=&2\varepsilon \sqrt{\frac{2n\sigma }{\pi }}\exp \left( -\widehat{f}\left(
l_{j}\right) \left( 1-\cos \left( l_{j}\theta _{j}\right) \right)
nr^{l_{j}}\left( 1+O\left( r\right) \right) \right) .
\end{eqnarray*}%
By definition of $l_{j}$, $1-\cos \left( l_{j}\theta _{j}\right) >0$, and
from (III) we have $\widehat{f}\left( l_{j}\right) >0$. Hence we can find a
positive constant $c$ such that $\left\vert I_{j}\right\vert \leq
2\varepsilon \sqrt{\frac{2n\sigma }{\pi }}\exp \left( -cnr^{l_{j}}\right) .$
In case $l$ is even and \thinspace $j=m$, the integral is over $\left[ \pi
-\varepsilon ,\pi \right] $, and the same estimate holds with $2\varepsilon $
replaced by $\varepsilon $. Since $k=n\mu _{f}\left( r\right) \approx
n^{\delta }$, using (I)(a) we find that $r\approx n^{\left( \delta -1\right)
/l},$ and hence, from (I) (b), $\lambda =\sqrt{n\sigma /2}\approx n^{\delta
/2}.$ So we conclude that $\lambda ^{N}\left\vert I_{j}\right\vert \leq
An^{\left( N+1\right) \delta /2}\exp \left( -cn^{1+l_{j}\left( \delta
-1\right) /l}\right) =An^{\left( N+1\right) \delta /2}\exp \left(
-cn^{b}\right) $, where $A$ is a positive constant, and $b=l_{j}\left(
\delta -\left( 1-l/l_{j}\right) \right) /l>0$. Hence $I_{j}=o\left( \lambda
^{-N}\right) .$
\end{proof}

The next lemma estimates $I_{j}$ when $k$ is not too large.

\begin{lemma}
\label{L4}Suppose that $0<\delta \leq 1-\frac{l}{l_{j}}$, $k\approx
n^{\delta },$ $k\equiv s \pmod{l} $, and $r>0$ is defined by 
$k=n\mu _{f}\left( r\right) $. Let $\gamma _{j}=\widehat{f}\left(
l_{j}\right) $ if $\delta =1-\frac{l}{l_{j}},$ and $\gamma _{j}=0$ if $%
\delta <1-\frac{l}{l_{j}}$. Then we find%
 \[
\lim_{n\rightarrow \infty }\mbox{\rm Re}I_{j}=\left\{ 
\begin{tabular}{ll}
$\left( -1\right) ^{s}e^{-2\gamma _{m}}$ & $\mbox{if }l\mbox{ is even and }%
j=m\mbox{,}$ \\ 
&  \\ 
$2e^{-\gamma _{j}\left( 1-\cos \left( l_{j}\theta _{j}\right) \right) }\cos
\left( s\theta _{j}-\gamma _{j}\sin \left( \theta _{j}l_{j}\right) \right) $
& $\mbox{in all other cases.}$%
\end{tabular}%
\right.
 \]
\end{lemma}

\begin{proof}
First assume that either $l$ is odd, or $j\neq m$. Using Lemma \ref{L1} and
Lemma \ref{L2}, we can write%
\begin{eqnarray*}
I_{j} &=&\sqrt{\frac{2n\sigma }{\pi }}\int\limits_{\widetilde{\theta }%
_{j}\left( r\right) -\varepsilon }^{\widetilde{\theta }_{j}\left( r\right)
+\varepsilon }\left( \frac{f\left( re^{i\theta }\right) }{f\left( r\right) }%
\right) ^{n}e^{-ik\theta }d\theta \\
&=&\sqrt{\frac{2n\sigma }{\pi }}e^{-ik\widetilde{\theta }_{j}\left(
r\right) }\int\limits_{-\varepsilon }^{\varepsilon }\left( \frac{f\left(
re^{i\widetilde{\theta }_{j}\left( r\right) }e^{it}\right) }{f\left(
r\right) }\right) ^{n}e^{-ikt}dt \\
&=&\sqrt{\frac{2n\sigma }{\pi }}\left( e^{-i\mu \widetilde{\theta }%
_{j}\left( r\right) }\frac{f\left( re^{i\widetilde{\theta }_{j}\left(
r\right) }\right) }{f\left( r\right) }\right) ^{n}\int\limits_{-\varepsilon
}^{\varepsilon }\left( \frac{f\left( re^{i\widetilde{\theta }_{j}\left(
r\right) }e^{it}\right) }{f\left( re^{i\widetilde{\theta }_{j}\left(
r\right) }\right) }\right) ^{n}e^{-ikt}dt \\
&=&\sqrt{\frac{2n\sigma }{\pi }}e^{-ik\theta _{j}+nA_{j}r^{l_{j}}+O\left(
nr^{l_{j}+1}\right) }\int\limits_{-\varepsilon }^{\varepsilon }\left( 
\frac{f\left( re^{it}\right) }{f\left( r\right) }\right)
^{n}e^{nr^{l_{j}}F\left( r,t\right) -ikt}dt \\
&=&\sqrt{\frac{2n\sigma }{\pi }}e^{-ik\theta _{j}+nA_{j}r^{l_{j}}+O\left(
nr^{l_{j}+1}\right) }\int\limits_{-\varepsilon }^{\varepsilon }e^{-n\sigma
t^{2}\phi \left( r,it\right) /2+nr^{l_{j}}F\left( r,t\right) }dt \\
&=&\sqrt{\frac{2n\sigma }{\pi }}e^{-ik\theta _{j}+nA_{j}r^{l_{j}}+O\left(
nr^{l_{j}+1}\right) }\frac{1}{\lambda }\int\limits_{-\lambda \varepsilon
}^{\lambda \varepsilon }e^{-x^{2}\phi \left( r,ix/\lambda \right)
+nr^{l_{j}}F\left( r,x/\lambda \right) }dx \\
&=&\frac{2}{\sqrt{\pi }}\exp \left( -ik\theta _{j}+nr^{l_{j}}\widehat{f}%
\left( l_{j}\right) \left( e^{il_{j}\theta _{j}}-1\right) +O\left(
nr^{l_{j}+1}\right) \right) \\
&&\times \int\limits_{-\lambda \varepsilon }^{\lambda \varepsilon
}e^{-x^{2}\phi \left( r,ix/\lambda \right) +nr^{l_{j}}F\left( r,x/\lambda
\right) }dx.
\end{eqnarray*}%
The assumptions imply that $nr^{l_{j}}\simeq n^{1+l_{j}\left( \delta
-1\right) /l}$ and the exponent of $n$ on the right side is $\leq 0,$ with
equality if and only if $\delta =1-\frac{l}{l_{j}}.$ Hence $\widehat{f}%
\left( l_{j}\right) nr^{l_{j}}\rightarrow \gamma _{j}.$

To evaluate the integral on the right side, use the fact that $\left\vert
\phi \left( r,it\right) -1\right\vert \leq 1/2$ for $\left\vert t\right\vert
\leq \varepsilon $ and $F\left( r,0\right) =0$ to find a constant $b$ such
that the integrand is bounded by $\exp \left( -x^{2}/2+b\right) $. Hence
using the dominated convergence theorem we find that the integral converges
to $\int\limits_{-\infty }^{\infty }e^{-x^{2}}dx=\sqrt{\pi }$, and the result
follows by taking the
real part of the limit.

The case $l$ even and $j=m$ is the same, except that the final integral is
over $\left( -\infty ,0\right) $, and the result follows because in this
case $\theta _{m}=\pi $ and $l_{m}$ is odd.
\end{proof}

Our final lemma in this section shows that as long as $k\rightarrow \infty $
at least like $n^{\delta }$ for some positive $\delta $, the contribution to 
$\widehat{f^{n}}\left( k\right) $ coming from $I_{0}$ is at least as
significant as the contribution coming from all other $I_{j}.$ This lemma is
also a direct consequence of the full asymptotic expansion for $\mbox{Re}%
I_{0}$ given in the proof of Theorem \ref{T0} below, but we derive it here
without appealing to \cite{ae}.

\begin{lemma}
\label{L5}Suppose $0<\delta <1$, $k\approx n^{\delta }$, and let $r>0$ be
defined by $k=n\mu _{f}\left( r\right) .$ Then $\lim\limits_{n\rightarrow
\infty }\mbox{\rm Re}I_{0}=1.$
\end{lemma}

\begin{proof}
We have 
\begin{eqnarray*}
\mbox{Re}I_{0} &=&\frac{1}{2}\sqrt{\frac{2n\sigma }{\pi }}%
\int\limits_{-\varepsilon }^{\varepsilon }\left( \frac{f\left( re^{i\theta
}\right) }{f\left( r\right) }\right) ^{n}e^{-ik\theta }d\theta \\
&=&\sqrt{\frac{n\sigma }{2\pi }}\int\limits_{-\varepsilon }^{\varepsilon
}e^{-\lambda ^{2}\theta ^{2}\phi \left( r,i\theta \right) }d\theta =\frac{1}{%
\sqrt{\pi }}\int\limits_{-\lambda \varepsilon }^{\lambda \varepsilon
}e^{-x^{2}\phi \left( r,ix/\lambda \right) }dx.
\end{eqnarray*}%
As before, the integrand is bounded by $e^{-x^{2}/2}$, and so taking the
limit as $n\rightarrow \infty $ we get our result.
\end{proof}

\section{Asymptotic estimates of coefficients}

Our first theorem extends Theorem 4.1 of \cite{ae} to the case that $%
k/n\rightarrow 0$ as $n\rightarrow \infty $, and Theorem 5.4 in the same
paper to the case $f^{\prime }\left( 0\right) =0.$

\begin{theorem}
\label{T0}Let $\epsilon _{0}=\max \left( 1-\frac{l}{l_{j}}:1\leq j\leq
m\right) ,$ and $\epsilon _{0}<\delta <1$. If $k\approx n^{\delta }$ and $r$
is defined by $k=n\mu _{f}\left( r\right) ,$ then as $n\rightarrow \infty $, 
$\widehat{f^{n}}\left( k\right) $ has the asymptotic expansion%
\begin{eqnarray*}
\widehat{f^{n}}\left( k\right) &\sim &\frac{\left( f\left( r\right) \right)
^{n}}{r^{k}\sqrt{2n\pi \sigma }}\left( 1+\sum_{\nu =1}^{\infty }\frac{%
c_{\nu }}{k^{\nu }}\right) \mbox{ } \\
\mbox{where }c_{\nu } &=&\frac{\left( -1\right) ^{\nu }\mu ^{\nu }}{\sigma
^{2\nu }\nu !}\left[ \frac{\partial ^{2\nu }}{\partial z^{2\nu }}\left( \phi
\left( r,z\right) ^{-\nu -1/2}\right) \right] _{z=0}.
\end{eqnarray*}
\end{theorem}

\begin{proof}
As in the proof of Lemma \ref{L5}, 
 \[
\mbox{Re}I_{0}=\sqrt{\frac{n\sigma }{2\pi }}\int\limits_{-\varepsilon
}^{\varepsilon }e^{-\lambda ^{2}\theta ^{2}\phi \left( r,i\theta \right)
}d\theta .
 \]

Using (VI), this expression has the asymptotic expansion%
 \[
\mbox{Re}I_{0}\sim \frac{1}{\lambda }\sqrt{\frac{n\sigma }{2}}\left(
1+\sum_{\nu =1}^{\infty }a_{\nu }\lambda ^{-2\nu }\right) =1+\sum_{\nu
=1}^{\infty }\frac{c_{\nu }}{k^{\nu }}.
 \]%
By Lemma \ref{L3}, $J$ and all $I_{j}$ are $o\left( \lambda ^{-\nu }\right) $
for every $\nu $, and the result follows from $\left( \ref{sum}\right) $.
\end{proof}

The next theorem provides asymptotic estimates for the coefficients $%
\widehat{f^{n}}\left( k\right) $ whenever $k$ is in the specified range.

\begin{theorem}
\label{T1}Let $0<\delta <1$, $k\approx n^{\delta },$ $k\equiv s \pmod{l} $,
 and let $r>0$ be defined by $k=n\mu _{f}\left( r\right) .$
Define 
 \[
\gamma _{j}=\left\{ 
\begin{tabular}{ll}
$0$ & $\mbox{if }l_{j}>\frac{l}{1-\delta }$ \\ 
&  \\ 
$\widehat{f}\left( l_{j}\right) $ & if $\ l_{j}=\frac{l}{1-\delta }$ \\ 
&  \\ 
$\infty $ & if $\ l_{j}<\frac{l}{1-\delta },$%
\end{tabular}%
\right.
 \]%
and%
 \[
\psi _{j}\left( s\right) =\left\{ 
\begin{tabular}{ll}
$\left( -1\right) ^{s}e^{-2\gamma _{m}}$ & $\mbox{if }l\mbox{ is even and }%
j=m,$ \\ 
&  \\ 
$2e^{-\gamma _{j}\left( 1-\cos \left( l_{j}\theta _{j}\right) \right) }\cos
\left( s\theta _{j}-\gamma _{j}\sin \left( l_{j}\theta _{j}\right) \right) $
& $\mbox{in all other cases,}$%
\end{tabular}%
\right.
 \]%
where we use the convention $e^{-\infty }=0.$

Then%
 \[
\lim_{n\rightarrow \infty }r^{k}\sqrt{2\pi n\sigma }\frac{\widehat{f^{n}}%
\left( k\right) }{\left( f\left( r\right) \right) ^{n}}=1+\sum%
\limits_{j=1}^{m}\psi _{j}\left( s\right) .
 \]
\end{theorem}

\begin{proof}
From $\left( \ref{sum}\right) $, $r^{k}\sqrt{2\pi n\sigma }\frac{\widehat{%
f^{n}}\left( k\right) }{\left( f\left( r\right) \right) ^{n}}=\mbox{Re}%
\left( I_{0}+\sum_{j=1}^{m}I_{j}+J\right) $. From Lemma \ref{L0}, $\mbox{Re}%
J\rightarrow 0$, and from Lemma \ref{L5}, $\mbox{Re}I_{0}\rightarrow 1$. If $%
l_{j}<l/\left( 1-\delta \right) $, Lemma \ref{L3} shows that $\mbox{Re}%
I_{j}\rightarrow 0.$ If $l_{j}\geq l/\left( 1-\delta \right) $, Lemma \ref%
{L4} shows that 
 \[
\begin{tabular}{ll}
$\mbox{Re}I_{j}\rightarrow 2\exp \left( -\gamma _{j}\left( 1-\cos \left(
l_{j}\theta _{j}\right) \right) \right) \cos \left( s\theta _{j}-\gamma
_{j}\sin \left( \theta _{j}l_{j}\right) \right) $ & if $l$ is odd or $j\neq
m $,%
\end{tabular}%
 \]%
and%
 \[
\begin{tabular}{ll}
$\mbox{Re}I_{m}\rightarrow \exp \left( -2\gamma _{m}\right) \left( -1\right)
^{s}$ & if $l$ is even.%
\end{tabular}%
 \]
\end{proof}

The next proposition is included as a curiosity: while in specific examples
it seems always easy to prove the non-negativity of the function there
described directly, I could not find a direct proof of the general case
(without appealing to the previous results).

\begin{proposition}
\label{p9}Let $L$ be a finite set of positive integers such that $\gcd
\left\{ k:k\in L\right\} =1.$ Let $l=\min L,$ and assume that $l>1.$ Let $m=%
\left[ l/2\right] .$ For each $j,1\leq j\leq m,$ let 
 \[
l_{j}=\min \left\{ k\in L:jk\not\equiv 0\pmod{l} \right\} .
 \]%
Fix $u\in \left\{ l_{1},l_{2},\ldots ,l_{m}\right\} $, and let%
 \[
\gamma _{j}=\gamma _{j}\left( t\right) =\left\{ 
\begin{tabular}{ll}
$0$ & if $l_{j}>u$ \\ 
$t$ & if $l_{j}=u$ \\ 
$\infty $ & if $l_{j}<u$%
\end{tabular}%
\right. .
 \]%
Let $\theta _{j}=2\pi j/l$, and let $s$ be an integer. Define%
 \[
\psi _{j}\left( s,t\right) =\left\{ 
\begin{tabular}{ll}
$\left( -1\right) ^{s}e^{-2\gamma _{m}}$ & $\mbox{if }l\mbox{ is even and }%
j=m,$ \\ 
&  \\ 
$2e^{-\gamma _{j}\left( 1-\cos \left( l_{j}\theta _{j}\right) \right) }\cos
\left( s\theta _{j}-\gamma _{j}\sin \left( l_{j}\theta _{j}\right) \right) $
& $\mbox{in all other cases.}$%
\end{tabular}%
\right.
 \]%
Then $1+\sum\limits_{j=1}^{m}\psi _{j}\left( s,t\right) \geq 0$ for all $%
t\geq 0$ and all $s$.
\end{proposition}

\begin{proof}
Consider the polynomial 
 \[
p\left( x\right) =1+\frac{1}{l}x^{l}+\sum\limits_{\begin{array}{c}
 k\in L,  \\ %
k\neq l,k\neq u
\end{array}}x^{k}+tx^{u}.
 \]%
Taking $\delta =1-l/u$ and $k=l\left\lfloor n^{1-l/u}\right\rfloor +s,$ the
asymptotic estimate for the coefficient of $x^{k}$ in $\left( p\left(
x\right) \right) ^{n}$ is, according to Theorem \ref{T1} ,%
 \[
\lim_{n\rightarrow \infty }r^{k}\sqrt{2\pi n\sigma }\frac{\widehat{p^{n}}%
\left( k\right) }{\left( p\left( r\right) \right) ^{n}}=1+\sum%
\limits_{j=1}^{m}\psi _{j}\left( s,t\right) .
 \]

Since $p\left( x\right) $ has no negative coefficients, the right side must
be non-negative.
\end{proof}

The following lemma will be needed in some of the examples in the next
section.

\begin{lemma}
\label{L6}Suppose $r\left( t\right) $ is defined by the equation $\mu
_{p}\left( r\left( t\right) \right) =t^{l}$ for $t>0$, and $r\left( 0\right)
=0$. Then $r\left( t\right) $ is analytic at $t=0.$
\end{lemma}

\begin{proof}
We can write $\mu _{p}\left( r\right) =r^{l}+r^{l+1}G\left( r\right) $ where 
$G\left( r\right) $ is analytic at $r=0$. Let $\rho \left( t\right) =r\left(
t\right) /t$, $t>0$, and $\rho \left( 0\right) =1.$ Then the equation
defining $r\left( t\right) $ gives us%
 \[
\left( \rho \left( t\right) \right) ^{l}+t\left( \rho \left( t\right)
\right) ^{l+1}G\left( r\left( t\right) \right) =1.
 \]%
Since $\mu _{p}$ is strictly increasing as a function of $r$, $r\left(
t\right) $ is clearly bounded as $t\rightarrow 0^{+}.$ Hence letting $t$
approach $0$ in the above equation, we find \ that $\rho \left( t\right) $
is continuous at $t=0$. If we let $F\left( t,\rho \right) =\rho ^{l}+t\rho
^{l+1}G\left( t\rho \right) $, we have $\frac{\partial F}{\partial \rho }%
\left( 0,1\right) =l>0.$ By the implicit function theorem, $\rho \left(
t\right) $ defined by $F\left( t,\rho \left( t\right) \right) =1$ exists and
is unique, and it is analytic at $t=0$. So $r\left( t\right) =t\rho \left(
t\right) $ is also analytic at $t=0.$
\end{proof}

\section{Examples}

We conclude with a few examples. Example \ref{ex1} is one of the simplest
cases, and Example \ref{ex2} was chosen because it does not seem immediately
evident that the corresponding function $\psi _{j}\left( s,t\right) $ is
non-negative. Example \ref{ex3} illustrates a case in
which the set $\left\{ l_{j}:1\leq j\leq \left\lfloor l/2\right\rfloor
\right\} $ has more than one element, and Example \ref{ex4} is used to show that sometimes the estimate 
provided by Theorem \ref{T1} is only an upper bound. 

\begin{example}
\label{ex1}Let $p\left( x\right) =1+\frac{1}{3}x^{3}+bx^{4}$, where $b$ is
a positive number. So we have (in the notation of Theorem \ref{T0}) $l=3$, $%
m=1,$ $l_{1}=4,$ and $\varepsilon _{0}=1/4.$ So if $1/4<\delta <1$, $%
k\approx n^{\delta }$, and $r$ is defined by $k=n\mu _{p}\left( r\right) $,
the asymptotic expansion of Theorem \ref{T0} applies. An explicit expression
for the dominant term as a function of $n$ and $k$ alone will in general
depend on the size of $\delta .$

We now choose $\delta =1/4,$ $n=k^{4},$ and we derive the asymptotic
estimate in terms of $k$ for the coefficient of $x^{k}$ in $\left( p\left(
x\right) \right) ^{k^{4}}$ when $k$ is large, as given by Theorem \ref{T1}.
The power expansion of $\mu _{p}\left( r\right) $ at $r=0$ is 
 \[
\mu _{p}\left( r\right) =r^{3}+4br^{4}+O\left( r^{6}\right) .
 \]%
In this case, the equation $n\mu _{p}\left( r\right) =k$ defining $r$ becomes 
$\mu _{p}\left( r\right) =1/k^{3}$, and so $r\simeq n^{-1/4}=1/k$. In order
to find the asymptotic behaviour of the terms $r^{k}$ and $\left( p\left(
r\right) \right) ^{n}$ appearing in Theorem \ref{T1}, we need to find higher
order terms for $r$ as a function of $k$. Using Lemma \ref{L6}, we expand $r$
as a power series in $1/k:$%
 \begin{equation}
r=c_{1}\frac{1}{k}+c_{2}\frac{1}{k^{2}}+\cdots .  \label{r}
 \end{equation}%
So $c_{1}=1$, and substituting $\left( \ref{r}\right) $ into $%
1/k^{3}=r^{3}+4br^{4}+O\left( r^{6}\right) $, we find $c_{2}=-4b/3.$ Hence%
\begin{eqnarray*}
r &=&\frac{1}{k}-\frac{4b}{3k^{2}}+O\left( \frac{1}{k^{3}}\right) \\
&=&\frac{1}{k}\left( 1-\frac{4b}{3k}+O\left( \frac{1}{k^{2}}\right) \right)
\end{eqnarray*}%
and we conclude that 
 \[
r^{k}\simeq \frac{1}{k^{k}}\exp \left( -\frac{4b}{3}\right)
 \]%
as $k\rightarrow \infty .$

In a similar way, expanding $p\left( r\right) $ we find%
 \[
p\left( r\right) =1+\frac{1}{3k^{3}}-\frac{b}{3k^{4}}+O\left( \frac{1}{k^{5}}%
\right) ,
 \]%
and so%
 \[
\left( p\left( r\right) \right) ^{k^{4}}\simeq \exp \left( \frac{k}{3}-\frac{%
b}{3}\right) .
 \]%
Finally, $\gamma _{1}=b$, $\theta _{1}=2\pi /3$, and using (I)(b) of Section
2, $\sqrt{2\pi n\sigma _{p}\left( r\right) }\simeq \sqrt{6\pi k}$, so the
estimate of Theorem \ref{T1} gives us%
\begin{eqnarray*}
\widehat{p^{k^{4}}}\left( k\right) &=&\sum\limits_{\begin{array}{c}
 3v+4w=k  \\ %
u+v+w=k^4
\end{array} }
\frac{\left( k^{4}\right) !}{u!v!w!}\left( \frac{1}{3}\right)
^{v}b^{w} \\
&\simeq &\frac{k^{k}e^{k/3}}{\sqrt{6\pi k}}e^{b}\left( 1+2e^{-3b/2}\cos
\left( \frac{2\pi k}{3}-\frac{b\sqrt{3}}{2}\right) \right) .
\end{eqnarray*}
\end{example}

\begin{example}
\label{ex2}Let $L=\left\{ 15,20,21\right\} .$ Then $l=15$, and $\left\{
l_{j}:1\leq j\leq 7\right\} =\left\{ 20,21\right\} .$ Choosing $u=21$ in
Proposition \ref{p9}, we have $l_{j}=u$ for $j=3,6$, and so the function%
\begin{eqnarray*}
g\left( s,t\right) &=&1+2e^{-t\left( 5-\sqrt{5}\right) /4}\cos \left( \frac{%
2\pi s}{5}-t\sqrt{\frac{5+\sqrt{5}}{8}}\right) \\
&&\mbox{ \ }+2e^{-t\left( 5+\sqrt{5}\right) /4}\cos \left( \frac{4\pi s}{5}%
-t\sqrt{\frac{5-\sqrt{5}}{8}}\right)
\end{eqnarray*}%
is non-negative for all $t\geq 0$ and all integers $s.$
\end{example}

\begin{example}
\label{ex3}Let $p\left( x\right) =1+\frac{1}{9}x^{9}+bx^{15}+cx^{25},$
where $b$ and $c$ are positive numbers. So we have $%
l=9,l_{1}=l_{2}=l_{4}=15,l_{3}=25$ and $\left\{ 1-l/l_{j}:1\leq j\leq
4\right\} =\left\{ 2/5,16/25\right\} .$ Hence $\varepsilon _{0}=16/25,$ and
Theorem \ref{T0} applies when $k\approx n^{\delta }$ with $16/25<\delta <1.$

To illustrate Theorem \ref{T1}, we first choose $\delta =16/25,$ $%
n=m^{25},k=m^{16},$ and we derive the asymptotic estimate of the coefficient
of $x^{k}$ in $\left( p\left( x\right) \right) ^{n}$ when $m$ is large. The
equation $n\mu _{p}\left( r\right) =k$ gives us $\mu _{p}\left( r\right)
=1/m^{9}$, and we need to expand $r$ as a power series in $1/m$ in order to
estimate $r^{k}$ and $\left( p\left( r\right) \right) ^{n}.$ Since $r\simeq
1/m,$ $p\left( r\right) \simeq 1,$ $k=m^{16}$ and $n=m^{25},$ we need to
expand $r$ up to terms of order $1/m^{17}$ and $p\left( r\right) $ up to
terms of order $1/m^{25}.$ After some extensive computations, we find:%
\[
r=\frac{1}{m}\left( 1-\frac{5b}{3m^{6}}+\frac{1}{81m^{9}}+\frac{275b^{2}%
}{9m^{12}}-\frac{53b}{243m^{15}}-\frac{25c}{9m^{16}}+O\left( \frac{1}{%
m^{17}}\right) \right)
\]%
and 
 \[
\log \left( mr\right) =-\frac{5b}{3m^{6}}+\frac{1}{81m^{9}}+\frac{175b^{2}%
}{6m^{12}}-\frac{16b}{81m^{15}}-\frac{25c}{9m^{16}}+O\left( \frac{1}{%
m^{17}}\right) .
 \]%
Hence%
 \[
r^{k}\simeq \frac{1}{m^{k}}\exp \left( -\frac{5b}{3}m^{10}+\frac{1}{81}%
m^{7}+\frac{175b^{2}}{6}m^{4}-\frac{16b}{81}m-\frac{25c}{9}\right) .
 \]%
Similarly,%
 \[
p\left( r\right) =1+\frac{1}{9m^{9}}-\frac{2b}{3m^{15}}+\frac{1}{81m^{18}}%
+\frac{50b^{2}}{3m^{21}}-\frac{16b}{81m^{24}}-\frac{16c}{9m^{25}}+O\left( 
\frac{1}{m^{26}}\right)
 \]%
and%
 \[
\log p\left( r\right) =\frac{1}{9m^{9}}-\frac{2b}{3m^{15}}+\frac{1}{%
162m^{18}}+\frac{50b^{2}}{3m^{21}}-\frac{10b}{81m^{24}}-\frac{16c}{9m^{25}%
}+O\left( \frac{1}{m^{26}}\right) .
 \]%
Hence%
 \[
\left( p\left( r\right) \right) ^{n}\simeq \exp \left( \frac{1}{9}m^{16}-%
\frac{2b}{3}m^{10}+\frac{1}{162}m^{7}+\frac{50b^{2}}{3}m^{4}-\frac{10b}{%
81}m-\frac{16c}{9}\right) .
 \]%
We can also (more easily) compute%
 \[
\sqrt{2\pi n\sigma _{p}\left( r\right) }\simeq 3m^{8}\sqrt{2\pi }
 \]%
and 
 \[
1+\sum\limits_{j=1}^{4}\psi _{j}\left( s\right) =1+2e^{-3c/2}\cos \left( 
\frac{2\pi s}{3}-\frac{\sqrt{3}}{2}c\right) ,
 \]%
and so we obtain the estimate%
\begin{eqnarray*}
\widehat{p^{n}}\left( k\right) &=&\sum\limits_{\begin{array}{c}
 9v+15w+25z=m^{16} \\ 
u+v+w+z=m^{25}
\end{array}}
\frac{\left( m^{25}\right) !}{u!v!w!z!}\left( \frac{1}{9}%
\right) ^{v}b^{w}c^{z} \\
&\simeq &\frac{m^{k}}{3\sqrt{2\pi k}}e^{P\left( m\right) }\left(
1+2e^{-3c/2}\cos \left( \frac{2\pi k}{3}-\frac{\sqrt{3}}{2}c\right) \right)
\\
\mbox{as }m &\rightarrow &\infty ,\mbox{ where }n=m^{25},k=m^{16},\mbox{ and 
} \\
P\left( m\right) &=&\frac{1}{9}m^{16}+bm^{10}-\frac{1}{162}m^{7}-\frac{%
25b^{2}}{2}m^{4}+\frac{2b}{27}m+c
\end{eqnarray*}
\end{example}

\begin{example}
\label{ex4}We use the same polynomial of Example \ref{ex3}, but we choose $%
\delta =1-9/15=2/5,$ $n=m^{15},k=m^{6}.$ Using the expansion for $r$ found
in Example \ref{ex3}, we obtain%
\begin{eqnarray*}
r^{k} &\simeq &\left( \frac{1}{m}\right) ^{m^{6}}e^{-5b/3}, \\
\left( p\left( r\right) \right) ^{n} &\simeq &\exp \left( \frac{m^{6}}{9}-%
\frac{b}{6}\right) , \\
\sqrt{2\pi n\sigma _{p}\left( r\right) } &\simeq &3m^{3}\sqrt{2\pi }, \\
1+\sum\limits_{j=1}^{4}\psi _{j}\left( s\right) &=&\left( 1+2\cos \frac{%
2\pi k}{3}\right) \left( 1+2e^{-3b/2}\cos \left( \frac{2\pi k}{9}+\frac{b%
\sqrt{3}}{2}\right) \right) .
\end{eqnarray*}%
Note that the last expression is identically zero if $k$ is not a multiple
of $3$. If $k$ is a multiple of $3$, Theorem \ref{T1} gives us 
\begin{eqnarray*}
\widehat{p^{n}}\left( k\right) &=&\sum\limits_{\begin{array}{c}
 9v+15w+25z=m^{6} \\ 
u+v+w+z=m^{15}
\end{array}}
\frac{\left( m^{15}\right) !}{u!v!w!z!}\left( \frac{1}{9}%
\right) ^{v}b^{w}c^{z} \\
&\simeq &\frac{m^{k}e^{k/9}}{\sqrt{2\pi k}}e^{b}\left( 1+2e^{-3b/2}\cos
\left( \frac{2\pi k}{9}+\frac{b\sqrt{3}}{2}\right) \right) \mbox{\ } \\
\mbox{as }m &\rightarrow &\infty ,\mbox{\ where }n=m^{15},k=m^{6},k\equiv
0\pmod{3} .
\end{eqnarray*}%
If $k$ is not a multiple of $3$, the factor $\frac{m^{k}e^{k/9}}{3\sqrt{%
2\pi k}}e^{b}$ is only an upper bound for $\widehat{p^{n}}\left( k\right) .$
\end{example}

% \appendix

% \section{}
% \label{}



% notes:

% \bibitem{label2}
% If there is a note, it should come last:
% \bibitem{label3} \note
% \end{subbibitems}

%%\bibitem{}

%%\end{thebibliography}


\begin{thebibliography}{1}

\bibitem{pcp}
V.~De\mbox{\ }Angelis.
\newblock Positivity conditions for polynomials.
\newblock {\em Ergodic Theory Dynam. Systems}, 14:\mbox{\ 23--51}, 1994.

\bibitem{ineq}
V.~De\mbox{\ }Angelis.
\newblock On the inequality $|p(z)|\leq p(|z|)$ for polynomials.
\newblock {\em Proc. Amer. Math. Soc.}, \mbox{123, No. 10}:\mbox{\ 2999--3007},
  1995.

\bibitem{ae}
V.~De\mbox{\ }Angelis.
\newblock Asymptotic expansions and positivity of coefficients for large powers
  of analytic functions.
\newblock {\em Int. J. Math. Math. Sci.}, \mbox{2003, No.
  16}:\mbox{1003--1025}, 2003.

\bibitem{Flaj}
P.~Flajolet and R.~Sedgewick.
\newblock {\em Analytic Combinatorics}.
\newblock preliminary version,
  http://algo.inria.fr/flajolet/Publications/book061023.pdf.

\bibitem{Gardy}
D.~Gardy.
\newblock Some results on the asymptotic behaviour of coefficients of large
  powers of functions.
\newblock {\em Discrete Mathematics}, 139:189--217, 1995.

\bibitem{Hay2}
W.~K. Hayman.
\newblock A characterization of the maximum modulus of functions regular at the
  origin.
\newblock {\em J. Anal. Math.}, 1:\mbox{\ 135--154}, 1951.

\bibitem{Hay}
W.K. Hayman.
\newblock A generalisation of \mbox{S}tirling's formula.
\newblock {\em J. Reine Angew. Math.}, \mbox{196}:\mbox{\ 67--95}, 1956.

\bibitem{MW1}
L.~Moser and M.~Wyman.
\newblock Asymptotic expansions.
\newblock {\em Canad. J. Math.}, \mbox{8}:\mbox{\ 225--233}, 1956.

\bibitem{MW2}
L.~Moser and M.~Wyman.
\newblock Asymptotic expansions ii.
\newblock {\em Canad. J. Math.}, \mbox{9}:\mbox{\ 194--209}, 1957.

\bibitem{Od2}
A.~Odlyzko.
\newblock Asymptotic enumeration methods.
\newblock {\em in Handbook of combinatorics, Elsevier, R. L. Graham, M.
  Groetschel, and L. Lovasz, eds.}, 2:\mbox{1063--1229}, 1995.

\bibitem{Od}
A.~Odlyzko.
\newblock Analytic methods in asymptotic enumeration,.
\newblock {\em Discrete math.}, \mbox{153, No.1--3}:\mbox{229--238}, 1996.

\bibitem{OR}
A.M. Odlyzko and L.B. Richmond.
\newblock On the unimodality of high convolutions of discrete distributions.
\newblock {\em Ann. Probab.}, 13:\mbox{\ 299--306}, 1985.

\end{thebibliography}
\end{document}